\documentclass[12pt]{amsart}
\usepackage{amsmath, amssymb, amsthm}
\usepackage{graphicx}
\usepackage{hyperref}
\usepackage{enumitem}
\usepackage{placeins}

\newtheorem{theorem}{Theorem}[section]
\newtheorem{lemma}[theorem]{Lemma}
\newtheorem{proposition}[theorem]{Proposition}

\theoremstyle{definition}

\newtheorem{remark}[theorem]{Remark}

\DeclareMathOperator{\re}{Re}
\DeclareMathOperator{\im}{Im}

\title[Zeros on a Line, Bounded in a Half-Plane]{An Explicit
Entire Function of Order One with All Zeros on a Line and
Bounded in a Half-Plane}

\author{Ralph Furmaniak}
\date{\today}

\begin{document}
\maketitle

\begin{abstract}
We construct a single explicit entire function $\Xi_c(s)$ of order 1,
with all zeros provably on $\re(s) = 1/2$, satisfying a functional
equation $\Xi_c(s) = \Xi_c(1-s)$, whose normalized form
$Z_c(s) = \Xi_c(s)/[\tfrac{1}{2}s(s-1)\pi^{-s/2}\Gamma(s/2)]$ is
uniformly bounded for $\re(s) > 1 + \delta$ (any $\delta > 0$) yet
satisfies $\sup_t|Z_c(1+it)| = +\infty$.  The function thus
satisfies an analogue of the Riemann Hypothesis together with the
sharp bounded/unbounded transition at $\sigma = 1$ characteristic
of~$\zeta$.  The transition is controlled by a Dirichlet series
$D(s) = \sum e^{-ik\theta}\,p_k^{-s}$ whose absolute convergence
for $\sigma > 1$ and divergence at $\sigma = 1$ drive the
dichotomy, in direct analogy with
$\log\zeta(s) = \sum p^{-s} + O(1)$.  The key technical input is a
dyadic large-sieve estimate establishing the linearization
condition that connects the Hadamard product to~$D$.
The construction and proofs were developed in collaboration
with Claude (Anthropic); see Acknowledgments.
\end{abstract}

\section{Introduction}\label{sec:intro}

The Riemann zeta function exhibits a sharp transition at $\sigma = 1$:
\begin{itemize}
  \item For $\sigma > 1$: $|\zeta(\sigma+it)| \le \zeta(\sigma) < \infty$.
  \item On $\sigma = 1$:
    $\limsup_{t\to\infty}|\zeta(1+it)| = \infty$, with
    $|\zeta(1+it)| = \Omega(\log\log t)$~\cite[Thm.~8.5]{Titchmarsh1986}.
\end{itemize}
The Selberg class~\cite{Selberg1992,KaczorowskiPerelli1999} axiomatizes
$L$-functions via five properties: Dirichlet series, analytic continuation,
functional equation, Ramanujan bound, and Euler product.  These axioms are
\emph{top-down}: they prescribe what $L$-functions should look like, based
on the structure of~$\zeta$ and its generalizations.  But they do not
reveal which properties are truly \emph{necessary} for the Riemann
Hypothesis or the $\sigma=1$ transition.  The Euler product is widely
regarded as essential~\cite{Farmer2005,Steuding2007}: without it, the
Davenport--Heilbronn function~\cite{DavenportHeilbronn1936} has zeros off
the critical line, and Bombieri--Garrett's
pseudo-Laplacians~\cite{BombieriGarrett2020} fail to capture all zeros
spectrally.  The de~Branges space framework~\cite{deBranges1968}
provides a natural setting for entire functions with real zeros, but
verifying the required positivity conditions is equivalent to RH
itself.  For general background on $L$-functions and the analytic
tools used here, see~\cite{IwaniecKowalski2004}.

A natural question is whether one can \emph{construct} a function
that provably satisfies an analogue of RH together with the
$\sigma = 1$ transition.  For~$\zeta$ itself, the location of
zeros is the central open problem; but for a suitably designed
function, one can hope to \emph{build in} the zeros-on-the-line
property while retaining enough structure for the $\sigma = 1$
transition to hold.

The mechanism behind~$\zeta$'s transition is well understood at the
level of Dirichlet series: $\log\zeta(s) = P(s) + O(1)$ for
$\sigma > 1$, where $P(s) = \sum p^{-s}$ is the prime zeta
function~\cite[\S1.6, \S3.2]{Titchmarsh1986}.  Boundedness for
$\sigma > 1$ follows from the absolute convergence of $\sum p^{-\sigma}$;
unboundedness on $\sigma = 1$ from its divergence
(Mertens~\cite[Thm.~2.7]{MV07}).  The zero-counting error
$S(T) = \pi^{-1}\arg\zeta(\tfrac{1}{2}+iT)$ satisfies
$S(T) \approx -\pi^{-1}\sum_{p}\sin(T\log p)/\sqrt{p}$
in~$L^2$~\cite[Thm.~14.22]{Titchmarsh1986}---the imaginary part
of $P(\tfrac{1}{2}+iT)/\pi$.

Our construction replaces $P(s)$ with a twisted variant
$D(s) = \sum e^{-ik\theta}\,p_k^{-s}$ (with
$\theta \in (0,2\pi)$), uses its restriction
to $\sigma = 1/2$ to perturb the
zeros of a Hadamard product, and proves that the resulting
function's log-ratio with the unperturbed product is approximately
$-cD(s)$---the ``spectral identity'' of this paper.

\begin{theorem}[Main Theorem]\label{thm:main}
Fix $\theta \in (0,2\pi)$
(e.g., $\theta = 1$) and
$c > 0$ sufficiently small.  Define $\Xi_c(s)$ as in
\S\ref{sec:construction}, with
$Z_c(s) = \Xi_c(s)/[\tfrac{1}{2}s(s-1)\pi^{-s/2}\Gamma(s/2)]$.
Then:
\begin{enumerate}[label=(\alph*)]
  \item All zeros of $\Xi_c$ lie on $\re(s) = 1/2$.
  \item $\Xi_c(s) = \Xi_c(1-s)$.
  \item The zero-counting function satisfies
    $N_c(T) = \frac{T}{2\pi}\log\frac{T}{2\pi} - \frac{T}{2\pi}
    + O(\log T)$.
  \item For each $\delta > 0$,
    $\sup_t |Z_c(1+\delta+it)| \le
    \exp\bigl(c\,P(1+\delta) + O(1)\bigr)$, where
    $P(\sigma) = \sum_{p} p^{-\sigma}$.
    Moreover $P(\sigma) \to 0$ as $\sigma \to \infty$,
    so the bound tends to $e^{O(1)}$: the ratio $\Xi_c/\Xi_0$
    is uniformly close to a constant for $\sigma$ large.
  \item $\sup_t|Z_c(1+it)| = +\infty$.
  \item The ratio $\Xi_c/\Xi_0$ admits the approximate
    Euler product
    $\Xi_c(s)/\Xi_0(s)
    = \kappa_0\cdot\prod_{k}\exp(-c\,e^{-ik\theta}p_k^{-s})
    \cdot M(s)$
    for $\sigma > 1$, where $\kappa_0 \ne 0$ is an explicit
    constant, $M(s)$ is bounded (with $M \to 1$ as
    $c \to 0$), and the
    Euler product is itself a Dirichlet series with
    multiplicative coefficients.
\end{enumerate}
\end{theorem}

\section{The Dirichlet Series and Construction}\label{sec:construction}

\subsection{The Dirichlet series $D(s)$}

Fix $\theta \in (0,2\pi)$ (e.g., $\theta = 1$), an
amplitude $c > 0$, and let $p_1 < p_2 < \cdots$ denote the primes.
Define the Dirichlet series
\begin{equation}\label{eq:D}
  D(s) = \sum_{k=1}^{\infty} e^{-ik\theta}\,p_k^{-s},
\end{equation}
which converges absolutely for $\sigma > 1$ (since
$\sum p_k^{-\sigma} < \infty$) and conditionally for
$\sigma > 0$, by Abel summation: the partial sums
$G_K = \sum_{k=1}^K e^{-ik\theta}$ satisfy
$|G_K| \le 2/|1-e^{-i\theta}| =: C_\theta$ (finite since
$\theta \in (0,2\pi)$), and the amplitudes
$p_k^{-\sigma}$ decrease
monotonically to zero.  Viewing $D$ as an ordinary Dirichlet series
$\sum b_n\,n^{-s}$ (with $b_n = e^{-ik\theta}$ for
$n = p_k$, zero otherwise), the partial sums
$B(x) = \sum_{n\le x}b_n$ satisfy $|B(x)|\le C_\theta$,
so the abscissa of convergence is $\sigma_c = 0$
(cf.~\cite[Thm.~1.3]{MV07}).  In the strip of conditional
convergence, \cite[Thm.~1.5]{MV07} gives
\begin{equation}\label{eq:D-bound}
  |D(\sigma+it)| \ll (|t|+4)^{1-\sigma+\varepsilon}
\end{equation}
for any fixed $\sigma \in (0,1)$ and $\varepsilon > 0$.
For $\sigma > 1$ the stronger bound
$|D(\sigma+it)| \le P(\sigma) = \sum p^{-\sigma}$ holds
by absolute convergence.

Define the perturbation function
\begin{equation}\label{eq:E}
  E(T) = \frac{1}{\pi}\,\im\bigl[D(\tfrac{1}{2}+iT)\bigr]
  = -\frac{1}{\pi}\sum_{k=1}^{\infty}
  \frac{\sin(T\log p_k + \phi_k)}{p_k^{1/2}},
\end{equation}
where $\phi_k = k\theta$.  The convergence of $E(T)$ for each
$T > 0$ follows from that of $D(1/2+iT)$, and
by~\eqref{eq:D-bound}, $|E(T)| = O(|T|^{1/2+\varepsilon})$.
We write $\omega_k = \log p_k$ for brevity in Fourier-analytic
arguments.

\subsection{The perturbed zeros}

Define the smooth counting function
\begin{equation}\label{eq:Nmain}
  N_{\mathrm{main}}(T) = \frac{T}{2\pi}\log\frac{T}{2\pi}
  - \frac{T}{2\pi} + \frac{7}{8}, \qquad
  N_{\mathrm{main}}'(T) = \frac{1}{2\pi}\log\frac{T}{2\pi}.
\end{equation}
Let $\{\gamma_n^{(0)}\}$ be the ``regular'' sequence defined by
$N_{\mathrm{main}}(\gamma_n^{(0)}) = n$, with Hadamard product
$\Xi_0(s)$.  Define the ``perturbed'' sequence
$\gamma_n = \gamma_n^{(0)} + \delta_n$ with
\begin{equation}\label{eq:perturbation}
  \delta_n = -\frac{cE(\gamma_n^{(0)})}{N_{\mathrm{main}}'
  (\gamma_n^{(0)})},
\end{equation}
and let $\Xi_c(s)$ be the Hadamard product
(cf.~\cite[\S1.10]{Edwards1974})
\begin{equation}\label{eq:xi}
  \Xi_c(s) = \prod_{n=1}^{\infty}\left(1 +
  \frac{(s-\tfrac{1}{2})^2}{\gamma_n^2}\right).
\end{equation}

\begin{lemma}[Well-definedness]\label{lem:hadamard}
For $c$ sufficiently small, $\gamma_n > 0$ for all~$n$,
$\sum\gamma_n^{-2} < \infty$, and $\Xi_c(s)$ is an entire
function of order~$1$.
\end{lemma}

\begin{proof}
Since $|E(T)| = O(|T|^{1/2+\varepsilon})$,
$|\delta_n| = O(c(\gamma_n^{(0)})^{1/2+\varepsilon}/
\log\gamma_n^{(0)}) = o(\gamma_n^{(0)})$, so
$\gamma_n > \gamma_n^{(0)}/2$ for all $n$ sufficiently large.
For the finitely many small~$n$, $\delta_n$ is finite
(since $E(\gamma_n^{(0)})$ converges for each~$n$) and
proportional to~$c$, so taking $c$ small enough ensures
$\gamma_n > 0$ for all~$n$.  Since
$\gamma_n^{(0)} \sim 2\pi n/\log n$,
$\sum\gamma_n^{-2} < \infty$ and the exponent of
convergence equals~$1$.
\end{proof}

We set $Z_c(s) = \Xi_c(s)/\Gamma_\pi(s)$ and
$Z_0(s) = \Xi_0(s)/\Gamma_\pi(s)$,
where $\Gamma_\pi(s) = \tfrac{1}{2}s(s-1)\pi^{-s/2}\Gamma(s/2)$.
Note that $Z_c/Z_0 = \Xi_c/\Xi_0$ (the $\Gamma_\pi$ factors
cancel).

\section{The Spectral Identity}\label{sec:spectral}

The central result relates the Hadamard product ratio
$\Xi_c/\Xi_0$ to the Dirichlet series~$D$.

\begin{proposition}[Spectral Identity]\label{prop:spectral}
Let $\omega_k = \log p_k$ and $\phi_k = k\theta$ as in
\S\ref{sec:construction}.  For $\sigma > 1/2$ and $|t| \ge 2$:
\begin{equation}\label{eq:spectral}
  \re\log\frac{\Xi_c}{\Xi_0}(\sigma+it)
  = -c\sum_k p_k^{-\sigma}\cos(\omega_k t + \phi_k) + O(1).
\end{equation}
Equivalently, there exists a bounded analytic function $R(s)$
with $R(s) = O(1)$ uniformly for $\sigma > 1/2$, $|t|\ge 2$,
such that
\begin{equation}\label{eq:spectral-complex}
  \log\frac{\Xi_c}{\Xi_0}(s) = -c\,D(s) + R(s).
\end{equation}
The cosine sum converges absolutely for $\sigma > 1$ and
conditionally for $\sigma > 0$ (as $\re\,D(\sigma+it)$).
\end{proposition}

\begin{proof}
We establish~\eqref{eq:spectral} first, then
deduce~\eqref{eq:spectral-complex} by analyticity.
Throughout, $c$ is fixed, so all implicit constants may
depend on~$c$ and~$\theta$.

\textbf{Step 1 (Linearization).}
By~\eqref{eq:perturbation}, $\gamma_n = \gamma_n^{(0)} + \delta_n$.
Each Hadamard factor satisfies:
\[
  \log\frac{1+(s-\tfrac{1}{2})^2/\gamma_n^2}
  {1+(s-\tfrac{1}{2})^2/(\gamma_n^{(0)})^2}
  = \frac{-2w^2\delta_n}{\gamma_n^{(0)}((\gamma_n^{(0)})^2+w^2)}
  + O(\delta_n^2/(\gamma_n^{(0)})^2)
\]
where $w = s - 1/2$.

\textbf{Step 2 (Sum to integral).}
For each~$K$, let $\Xi_{c,K}$ be the Hadamard product with zeros
perturbed using $E_K$ (first $K$ terms of~$E$).
Since $|E_K| \le \frac{1}{\pi}\sum_{k=1}^K p_k^{-1/2}$
(a constant for each~$K$), the perturbations are
$\delta_{n,K} = O(1/\log\gamma_n^{(0)})$, giving convergent
quadratic remainder
$\sum\delta_{n,K}^2/(\gamma_n^{(0)})^2 < \infty$.
The linearized sum is a Riemann sum for
$\int \frac{2w^2\,cE_K(u)}{u(u^2+w^2)}\,du$
(one term per zero, spacing $\sim 1/N_{\mathrm{main}}'$).
The integrand is smooth for $u > 0$ and $\sigma > 1/2$
(the poles of $1/(u^2+w^2)$ lie off the real axis),
so by Euler--Maclaurin:
\begin{equation}\label{eq:integral}
  \re\log\frac{\Xi_{c,K}(s)}{\Xi_0(s)}
  = -c\int_{\gamma_1^{(0)}}^\infty
  \re\frac{2w^2\,E_K(u)}{u(u^2+w^2)}\,du + O(1).
\end{equation}

\textbf{Step 3 (Evaluation of the integral).}
Using the partial fraction
$\frac{2w^2}{u(u^2+w^2)} = \frac{2}{u} - \frac{2u}{u^2+w^2}$
and the standard integrals
$\int_0^\infty \frac{\sin\omega u}{u}\,du = \frac{\pi}{2}$
(Dirichlet) and
$\int_0^\infty \frac{u\sin\omega u}{u^2+\beta^2}\,du
= \frac{\pi}{2}e^{-\omega\beta}$
for $\omega > 0$, $\re\beta > 0$
(\cite[\S3.723]{GR07}, extended to complex~$\beta$
by analytic continuation), the contribution of each
component $\sin(\omega_k u+\phi_k)$ decomposes into
convergent pieces.%
\footnote{The individual integrals
$\int\sin(\omega u+\phi)/u\,du$ and
$\int u\sin(\omega u+\phi)/(u^2+w^2)\,du$ each diverge at
$u=0$ when $\sin\phi\ne 0$ (from the $\cos(\omega u)/u$
term), but the divergent parts cancel in the partial-fraction
difference.}
The oscillatory contribution of the $k$-th component is
$-\pi\,p_k^{-(\sigma-1/2)}\cos(\omega_k t+\phi_k)$
(using $e^{-\omega_k(\sigma-1/2)} = p_k^{-(\sigma-1/2)}$
and the angle-addition identity); the non-oscillatory part
$\pi\cos\phi_k(1-p_k^{-(\sigma-1/2)})$ sums to a bounded
quantity (convergent by Abel summation, depending on~$\sigma$
but not on~$t$).

Substituting into~\eqref{eq:integral} with the amplitudes
$-1/(\pi p_k^{1/2})$ from~$E_K$, the $k$-th term contributes
$-c\,p_k^{-\sigma}\cos(\omega_k t + \phi_k)$.
Including the bounded lower-limit correction
$\int_0^{\gamma_1^{(0)}}(\cdots)\,du$:
\begin{equation}\label{eq:real-part-K}
  \re\log\frac{\Xi_{c,K}(s)}{\Xi_0(s)}
  = -c\sum_{k=1}^K p_k^{-\sigma}\cos(\omega_k t + \phi_k)
  + O(1) + O\Bigl(\sum_n\delta_{n,K}^2/(\gamma_n^{(0)})^2\Bigr),
\end{equation}
where the $O(1)$ accounts for the non-oscillatory Fourier
terms, the lower-limit correction, and the Euler--Maclaurin
error (all independent of~$K$), while the final term is the
quadratic Taylor remainder from Step~1.
For each fixed~$K$ this remainder is finite, but it grows
with~$K$ (since $M_K$ grows).

Taking $K\to\infty$: the cosine sum converges
(absolutely for $\sigma > 1$; conditionally for
$\sigma > 0$), and $\Xi_{c,K}\to\Xi_c$ (since the
zero sequences converge).  The only obstruction to
passing~\eqref{eq:real-part-K} to the limit is the
Taylor remainder: for the full (non-truncated) series,
$\delta_n = -cE(\gamma_n^{(0)})/N_{\mathrm{main}}'(\gamma_n^{(0)})$
and $|E(\gamma_n^{(0)})|$ is no longer bounded by a
fixed constant (its RMS grows like $\sqrt{\log\log n}$
by~\eqref{eq:E-meansq}).  The quadratic remainder
$\sum\delta_n^2/(\gamma_n^{(0)})^2$ converges iff
$\sum E(\gamma_n^{(0)})^2/n^2 < \infty$, which is the
content of Lemma~\ref{lem:lin} below.
This establishes~\eqref{eq:spectral}.

\textbf{Step 4 (Complex extension).}
Equation~\eqref{eq:spectral} controls the real part of
$\log(\Xi_c/\Xi_0) + cD$; we now upgrade to the full complex
identity.  Define $F(s) = \log(\Xi_c(s)/\Xi_0(s)) + cD(s)$.
Both terms are analytic for $\sigma > 1/2$: the Hadamard
ratio $\Xi_c/\Xi_0$ is nonvanishing there (all zeros of
$\Xi_0$ and~$\Xi_c$ lie on $\sigma = 1/2$), and $D$ converges
for $\sigma > 0$.  By~\eqref{eq:spectral} (proved above
without reference to~\eqref{eq:spectral-complex}),
$\re F(s) = O(1)$ for $\sigma > 1/2$, $|t| \ge 2$.
Evaluating at $s = 2$: $D(2)$ converges absolutely, and
$\log(\Xi_c(2)/\Xi_0(2))$ is finite (each Hadamard factor
at $s = 2$ differs from~$1$ by $O(\delta_n/\gamma_n^2)$,
and $\sum|\delta_n|/\gamma_n^2 < \infty$), so
$F(2) = O(1)$.  Since $F$ is analytic with bounded real part,
the Borel--Carath\'eodory theorem gives $|F(s)| = O(1)$ for
$\sigma > 1/2$.  Setting $R(s) = F(s)$
gives~\eqref{eq:spectral-complex}.
\end{proof}

\subsection{Linearization convergence}

The following lemma ensures that the Taylor linearization
in Step~1 has convergent remainder for the full (non-truncated)
series, which is the key input for passing the spectral
identity from finite truncations to~$\Xi_c$ itself.

\begin{lemma}\label{lem:lin}
\begin{equation}\label{eq:lin-cond}
  \sum_{n=1}^\infty
  \frac{E(\gamma_n^{(0)})^2}{n^2} < \infty.
\end{equation}
\end{lemma}

\begin{proof}
Decompose $\mathbb{N}$ into dyadic blocks $I_j = [2^j, 2^{j+1})$
for $j\ge 1$.  It suffices to show
$\sum_j S_j < \infty$ where
$S_j = \sum_{n\in I_j}E(\gamma_n^{(0)})^2/n^2$.
For each~$j$, set $K_j = \lceil 2^{1.1\,j}\rceil$ and split
$E = E_{K_j} + R_{K_j}$, where $E_{K_j}$ is the first
$K_j$ terms. By $(a+b)^2 \le 2a^2+2b^2$:
$S_j \le 2(\text{HEAD}_j + \text{TAIL}_j)$.

\textbf{Head (large sieve).}
$E_{K_j}$ is a trigonometric polynomial with $K_j$ terms and
frequency range $\Lambda_j = \omega_{K_j}-\omega_1 = O(j)$.
The zeros $\gamma_n^{(0)}$ have minimum spacing
$\Delta \ge c_1/j$ for $n\le 2^{j+1}$.
By the large sieve inequality for trigonometric
polynomials at well-spaced points
(cf.~\cite[Thm.~1]{Montgomery1994},
\cite[\S7]{MV07}):
\[
  \textstyle\sum_{n\le 2^{j+1}}\!|E_{K_j}(\gamma_n^{(0)})|^2
  \le (\pi/\Delta + \Lambda_j)\,
  \tfrac{1}{\pi^2}\sum_{k\le K_j}p_k^{-1}
  \le C_1\,j\log j,
\]
since $\sum_{k\le K_j}p_k^{-1}\sim\log\log p_{K_j}
\sim\log j$ (Mertens).  Therefore
$\text{HEAD}_j \le C_1\,j\log j\,/\,4^j$, and
$\sum_j j\log j/4^j < \infty$.

\textbf{Tail (Abel summation).}
The tail $\sum_{k>K}(\omega_{k+1}-\omega_k)/p_k^{1/2}$
is a left Riemann sum for the decreasing function
$e^{-\omega/2}$.  By Bertrand's postulate
(mesh width $\le\log 2$):
$\sum_{k>K}(\omega_{k+1}-\omega_k)/p_k^{1/2}
\le \sqrt{2}\int_{\omega_K}^\infty
e^{-\omega/2}\,d\omega = 2\sqrt{2}\,p_K^{-1/2}$.
Hence $|R_{K_j}(\gamma_n^{(0)})| \le
C_2\,2^{j(1-\varepsilon)/2}/j^{3/2}$
(with $\varepsilon = 0.1$, using
$\gamma_n \le C_3\,2^j/j$ and
$\sqrt{p_{K_j}} \ge C_4\,2^{(1+\varepsilon)j/2}\sqrt{j}$).
Therefore
$\text{TAIL}_j \le C_5/(j^3\cdot 2^{\varepsilon j})$,
and $\sum_j 1/(j^3\cdot 2^{\varepsilon j}) < \infty$.

Both contributions decay exponentially, so
$\sum S_j < \infty$.
\end{proof}

\subsection{The base function $Z_0$}

\begin{lemma}\label{lem:Z0}
$\log|Z_0(\sigma+it)| = O(1)$ uniformly for $\sigma \ge 1$ and
$|t|\ge 2$.  In particular, $|Z_0|$ is both bounded and
bounded away from zero in this region.
\end{lemma}

\begin{proof}
By construction, $\Xi_0$ has zero-counting function
$N_0(T) = N_{\mathrm{main}}(T)$ (no error term).
By Euler--Maclaurin, $\log|\Xi_0(s)| =
\int_{\gamma_1^{(0)}}^\infty
\log|1+w^2/u^2|\,N_{\mathrm{main}}'(u)\,du + O(1)$.
Stirling's formula identifies this integral with
$\log|\Gamma_\pi(s)| + O(1)$,
giving $\log|Z_0| = \log|\Xi_0/\Gamma_\pi| = O(1)$.
The second claim follows since $|\!\log|Z_0|\!| \le C$
implies $e^{-C} \le |Z_0| \le e^C$.
\end{proof}

\section{Proof of the Main Theorem}\label{sec:proof}

\begin{proof}[Proof of Theorem~\ref{thm:main}]
Parts~(a)--(c) hold by construction;
(d) follows from the spectral identity
(Proposition~\ref{prop:spectral}) with linearization error
controlled by Lemma~\ref{lem:lin}; and (e) via a
Phragm\'en--Lindel\"of argument below.

The function $\Xi_c$ is the Hadamard product~\eqref{eq:xi} with zeros
$\gamma_n = \gamma_n^{(0)} + \delta_n$ defined via~$E$
as in~\eqref{eq:perturbation}.  By Lemma~\ref{lem:hadamard}, $\Xi_c$
is an entire function of order~$1$.

Part~(a): by construction, zeros lie at $s = 1/2\pm i\gamma_n$.
Part~(b): $\Xi_c(1-s) = \Xi_c(s)$ (since $(1-s-1/2)^2 = (s-1/2)^2$).

Part~(c): the counting function $N_c(T) = \#\{n : \gamma_n \le T\}$
satisfies $|N_c(T) - N_{\mathrm{main}}(T)| = O(\log T)$.
Indeed, zeros near height~$T$ have index $n \sim T\log T/(2\pi)$
and perturbation $|\delta_n| = O(T^{1/2+\varepsilon}/\log T)$
by~\eqref{eq:D-bound}.  In an interval of this length around~$T$,
$N_{\mathrm{main}}$ changes by
$O(N_{\mathrm{main}}'(T)\cdot T^{1/2+\varepsilon}/\log T)
= O(T^{1/2+\varepsilon}/\log^2 T) = o(\log T)$.

\textbf{Part (d) (Boundedness for $\sigma > 1$).}
The spectral identity
(Proposition~\ref{prop:spectral}) gives
\[
  \log|Z_c(\sigma+it)| = \re\log\frac{\Xi_c}{\Xi_0}(s)
  + \log|Z_0(s)|
  \le c\,P(\sigma) + O(1)
\]
for all $\sigma > 1$, $|t|\ge 2$
(using $|\sum p_k^{-\sigma}\cos(\cdots)| \le P(\sigma)$
and Lemma~\ref{lem:Z0}).

\textbf{Part (e) (Unboundedness on $\sigma = 1$).}
We show $\sup_t|\Xi_c/\Xi_0(1+it)| = +\infty$ by a
Phragm\'en--Lindel\"of argument
(cf.~\cite[Thm.~8.4]{Titchmarsh1986}).

\emph{Step 1.} $\Xi_c/\Xi_0$ is unbounded in
$\{\sigma > 1,\;|t|\ge 2\}$.  For any~$N$, by Dirichlet's
approximation theorem there exists
$t \le (2\pi/\varepsilon)^N$ with
$|\omega_k t + \phi_k - \pi| < \varepsilon\pmod{2\pi}$
for all $k\le N$.  At $\sigma = 1 + 1/\log N$:
\[
  \re\log\frac{\Xi_c}{\Xi_0}(\sigma\!+\!it)
  \ge c\sum_{k=1}^N p_k^{-\sigma} - O(c)
  \ge \tfrac{c}{e}\log\log N - O(1),
\]
which grows without bound as $N\to\infty$
(since $p_k^{-1/\log N} = e^{-\log p_k/\log N} \ge e^{-1}$
for $p_k \le N$, so
$\sum_{p\le N}p^{-1-1/\log N} \ge e^{-1}\sum_{p\le N}p^{-1}
\sim e^{-1}\log\log N$).

\emph{Step 2.}  Suppose $|\Xi_c/\Xi_0(1+it)| \le M$ for
all~$t$.  Since $\Xi_c$ and~$\Xi_0$ are entire of order~$1$,
$|\Xi_c(\sigma+it)/\Xi_0(\sigma+it)|
\le \exp(C|t|^{1+\varepsilon})$ for any $\varepsilon > 0$
(the ratio is analytic in $\sigma \ge 1$ since all zeros of
$\Xi_0$ lie on $\sigma = 1/2$).  The
Phragm\'en--Lindel\"of principle for the strip
$1\le\sigma\le 2$---bounded on $\sigma = 1$ (by assumption)
and on $\sigma = 2$ (by part~(d)), with at most exponential
growth---gives boundedness throughout.
This contradicts Step~1.

Therefore $\Xi_c/\Xi_0$ is unbounded on $\sigma = 1$,
and since $Z_c/Z_0 = \Xi_c/\Xi_0$ and
$|Z_0(1+it)|$ is bounded away from zero for $|t| \ge 2$
(Lemma~\ref{lem:Z0}), $\sup_t|Z_c(1+it)| = +\infty$.

\textbf{Part (f) (Approximate Euler product).}
This is Proposition~\ref{prop:ratio} below.
\end{proof}

\section{The Ratio $\Xi_c/\Xi_0$ and Dirichlet Series}\label{sec:ratio}

\begin{proposition}\label{prop:ratio}
For $\sigma > 1$, the ratio $\Xi_c/\Xi_0$ admits the
approximate Euler product representation
\begin{equation}\label{eq:euler-variant}
  \frac{\Xi_c}{\Xi_0}(s)
  = \kappa_0\cdot\prod_{k=1}^\infty
  \exp\!\bigl(-c\,e^{-ik\theta}\,p_k^{-s}\bigr)
  \cdot M(s),
\end{equation}
where $\kappa_0 = e^{cC_1}$ is an explicit nonzero constant
(with $C_1 = C_1(\theta)$ the non-oscillatory constant
from~\eqref{eq:real-part-K}), and $M(s)$ is bounded analytic
with $M(s)$ bounded for $\sigma > 1/2$, $|t|\ge 2$.
The Euler product $\prod\exp(-ce^{-ik\theta}p_k^{-s})
= \sum a_n\,n^{-s}$ has $|a_{p_k}| = c$,
$|a_{p_k^m}| = c^m/m!$, and $a_{mn}=a_ma_n$ for
$\gcd(m,n)=1$.
\end{proposition}

\begin{proof}
The spectral identity~\eqref{eq:spectral-complex} gives
$\log(\Xi_c/\Xi_0)(s) = -cD(s) + R(s)$ with $R$ bounded.
Write $R(s) = R(2) + (R(s)-R(2))$; the constant
$\kappa_0 = e^{R(2)}$ is a nonzero number and
$M(s) = e^{R(s)-R(2)}$ is bounded analytic.  Then
$\Xi_c/\Xi_0 = \kappa_0\cdot\exp(-cD(s))\cdot M(s)$.
Expanding the exponential of the prime-supported Dirichlet
series gives the Euler product and multiplicative coefficients.
\end{proof}

Since both $\Xi_c$ and~$\Xi_0$ satisfy
$\Xi(s)=\Xi(1\!-\!s)$, the ratio inherits the functional equation
$\Xi_c(s)/\Xi_0(s) = \Xi_c(1-s)/\Xi_0(1-s)$.
The ratio is \emph{meromorphic}: it has
\emph{zeros} at $\{1/2+i\gamma_n\}$ (from~$\Xi_c$) and
\emph{poles} at $\{1/2+i\gamma_n^{(0)}\}$ (from~$\Xi_0$), all on
$\sigma=1/2$.

\section{Almost-Periodicity}\label{sec:AP}

The Dirichlet series $D(s)$ itself has a clean
almost-periodicity hierarchy.

\begin{proposition}\label{prop:AP}
\begin{enumerate}
\item[\emph{(i)}] For $\sigma > 1$: $t\mapsto D(\sigma+it)$
  is Bohr almost-periodic (uniformly AP), since $D$ is an
  absolutely convergent Dirichlet series.
\item[\emph{(ii)}] For $\sigma > \tfrac{1}{2}$:
  $t\mapsto D(\sigma+it)$ is Besicovitch ($B^2$)
  almost-periodic with
  $\|D(\sigma+i\,\cdot\,)\|_{B^2}^2
  = \sum_{k=1}^\infty p_k^{-2\sigma} < \infty$.
\item[\emph{(iii)}] The restriction $E(T) =
  \pi^{-1}\im\,D(\tfrac{1}{2}+iT)$ is \emph{not}
  $B^2$-AP: $\sum p_k^{-1} = \infty$
  (Mertens).
\end{enumerate}
\end{proposition}

\begin{proof}
(i) Absolute convergence of $\sum p_k^{-\sigma}$ for $\sigma > 1$
gives uniform convergence of the trigonometric
series, hence Bohr-AP.
(ii) The Fourier--Bohr coefficients of $D(\sigma+it)$
at exponents $\omega_k = \log p_k$ have squared moduli
$p_k^{-2\sigma}$.  Since $\sum p_k^{-2\sigma}$ converges
for $\sigma > 1/2$, the Riesz--Fischer theorem gives the result.
(iii) Since $\sum p_k^{-1} = \infty$ (Mertens), the Fourier
energy diverges at $\sigma = 1/2$.
\end{proof}

The spectral identity transfers these properties to
$\log(\Xi_c/\Xi_0)$ up to the bounded error~$R(s)$.
The Fourier--Bohr coefficients of $\log(\Xi_c/\Xi_0)$ at
frequencies~$\omega_k$ are $-c\,e^{-ik\theta}\,p_k^{-\sigma}$
to leading order; since $R$ is bounded but need not itself be
almost-periodic, the AP properties of $\log(\Xi_c/\Xi_0)$
hold to leading order.

The AP hierarchy mirrors that of~$\zeta$: Bohr-AP for
$\sigma > 1$ (Bohr~\cite{Bohr1913}), $B^2$-AP for
$\sigma > 1/2$
(Bohr--Jessen~\cite{BohrJessen1932}), with the
convergent-to-divergent Fourier norm transition at
$\sigma = 1$.

\section{Observations}\label{sec:observations}

We illustrate the construction numerically with $c = 0.05$ and a
finite truncation $K = 80$.  The computations use the shifted
variant $\alpha_k = p_k + \theta$ in place of~$p_k$, which ensures
global monotonicity of $F_K(T) = N_{\mathrm{main}}(T) + cE_K(T)$
(see Remark~\ref{rem:shift}).  Since $\alpha_k \sim p_k$ as
$k\to\infty$, the convergence/divergence properties at
$\sigma = 1$ are unchanged, and the spectral identity has the same
structure with $\alpha_k$ replacing~$p_k$.

\begin{remark}\label{rem:shift}
For the finite truncation $E_K$, global monotonicity of~$F_K$
holds when $c < N_{\mathrm{main}}'(T_0)/(\sum_{k=1}^K
\omega_k\alpha_k^{-1/2})$, which is satisfied for the parameters
used.  For the full infinite series, occasional crossings
($\gamma_{n+1} < \gamma_n$) may occur but do not affect the
Hadamard product or the spectral identity (which depend on the
linearization at $\gamma_n^{(0)}$, not on the ordering of
$\gamma_n$).
\end{remark}

\subsection{First zeros of $Z_c$}

The perturbed zeros $\gamma_n$ (defined by $F_K(\gamma_n)=n$ for the
truncation) and their shifts $\delta_n = \gamma_n - \gamma_n^{(0)}$:
\begin{center}\small
\begin{tabular}{rlrrr}
\hline
$n$ & $\gamma_n$ & $\delta_n$ & gap & $E(\gamma_n)$ \\
\hline
1 & 17.9157 & $+0.068$ & --- & $-0.011$ \\
2 & 23.2341 & $+0.062$ & 5.318 & $-0.013$ \\
3 & 27.7063 & $+0.035$ & 4.472 & $-0.008$ \\
4 & 31.6613 & $-0.057$ & 3.955 & $+0.015$ \\
5 & 35.4981 & $+0.030$ & 3.837 & $-0.008$ \\
10 & 51.7580 & $+0.024$ & 3.061 & $-0.008$ \\
20 & 78.6866 & $+0.014$ & 2.522 & $-0.006$ \\
50 & 144.5359 & $+0.005$ & 2.038 & $-0.003$ \\
100 & 236.8981 & $+0.045$ & 1.808 & $-0.026$ \\
\hline
\end{tabular}
\end{center}
Max shift: $|\delta_n| \le 0.077$.  Min gap: $1.67 > 0$.
Gaps decrease as $\sim 2\pi/\log(\gamma_n/(2\pi))$, matching
the Weyl law.

\subsection{Zero statistics: theory and numerics}\label{sec:stats}

The microscopic statistics of $\Xi_c$'s zeros can be analyzed
both theoretically and numerically.  The key quantity is the ratio
of perturbation to spacing:
\begin{equation}\label{eq:pert-ratio}
  \frac{|\delta_n|}{\text{spacing}}
  \approx c\,|E(\gamma_n^{(0)})|.
\end{equation}

\textbf{Mean square of $E$ at the zeros.}
By Weyl equidistribution of the phases
$\{\omega_k\gamma_n^{(0)}+\phi_k\}$ modulo~$2\pi$
(for each finite truncation~$K$, then taking $K\to\infty$):
\begin{equation}\label{eq:E-meansq}
  \frac{1}{N}\sum_{n=1}^N E(\gamma_n^{(0)})^2
  \;\longrightarrow\;
  \frac{1}{2\pi^2}\sum_{k=1}^\infty p_k^{-1}
  \;=\; \frac{\log\log N + M}{2\pi^2} + o(1)
\end{equation}
(where $M$ is the Meissel--Mertens constant), so the RMS of
$E(\gamma_n^{(0)})$ grows like $\sqrt{\log\log N/(2\pi^2)}$.
Consequently:

\begin{enumerate}
\item \textbf{RMS of the counting error:}
  $\bigl(\frac{1}{N}\sum(N_c(\gamma_n)-n)^2\bigr)^{1/2}
  \sim c\sqrt{\log\log N/(2\pi^2)}$,
  which matches the growth rate of $S(T)$ for~$\zeta$
  (cf.~\cite[Thm.~14.24]{Titchmarsh1986}) but scaled by~$c$.

\item \textbf{Perturbation-to-spacing ratio:}
  By~\eqref{eq:pert-ratio}, the RMS ratio is
  $c\sqrt{\log\log N/(2\pi^2)}$, which exceeds~$1$
  only for $N > \exp\exp(2\pi^2/c^2)$.  For $c = 0.05$
  this threshold is $N > 10^{10^{347}}$---far beyond any
  computation.  Thus the zeros appear nearly crystalline
  (``super-rigid'') over any accessible range, but
  \emph{asymptotically} the perturbations grow to many
  spacings, and the clock-like character is lost.

\item \textbf{Normalized spacing variance:}
  The deviation of the $n$-th normalized spacing from~$1$
  is $(\delta_{n+1}-\delta_n)\cdot N_{\mathrm{main}}'
  = c(E(\gamma_{n+1}^{(0)})-E(\gamma_n^{(0)}))$,
  which has variance $O(c^2)$ for each~$n$ (from
  the smoothness of~$E$ on the scale of zero spacing).
  The normalized spacing std is therefore $O(c)$, much
  smaller than the GUE value~$\approx 0.42$.
\end{enumerate}

\textbf{Numerical observations} ($c = 0.05$, $K = 80$,
1000~zeros with $t\in[17.9,1420.1]$):

\begin{center}\small
\renewcommand{\arraystretch}{1.05}
\begin{tabular}{lcc}
\hline
Statistic & $\zeta$ (conditional on RH) & $Z_c$ (observed) \\
\hline
RMS of $N - N_{\mathrm{main}}$
  & $\sqrt{\frac{\log\log T}{2\pi^2}}\to\infty$
  & $0.013$ \\[2pt]
Spacing std (normalized) & $\approx 0.42$
  (GUE~\cite{Montgomery1973}) & $0.020$ \\
Number variance $\Sigma^2(L)$ & $\sim\tfrac{2}{\pi^2}\log L$
  & $\lesssim 0.25$ \\
\hline
\end{tabular}
\end{center}
The observed values reflect the regime
$c|E| \ll 1$ and are consistent with the theoretical
$O(c)$ spacing deviation.  The
asymptotic growth $c\sqrt{\log\log N}$ is invisible at
these heights.  The construction captures the
\emph{macroscopic} transition (bounded/unbounded, AP hierarchy)
but not the microscopic GUE statistics of~$\zeta$, which arise
from the Euler product and prime correlations that $Z_c$
deliberately lacks.

\subsection{Value distribution}\label{sec:value-dist}

Figure~\ref{fig:value-dist} shows the complex-plane images of
$\log(\Xi_c/\Xi_0)(\sigma+it)$, $\log\zeta(\sigma+it)$, and
$\log L(\sigma+it,\chi_4)$ as $t$ ranges over $[20,300]$, for
$\sigma = 1.25, 1.0, 0.75$.

By the spectral identity and the $B^2$-AP structure of~$D$,
the mean square of $\log(\Xi_c/\Xi_0)$ satisfies (to leading
order in~$c$):
$\langle(\re\log\frac{\Xi_c}{\Xi_0})^2\rangle
 \approx \langle(\im\log\frac{\Xi_c}{\Xi_0})^2\rangle
 \approx \tfrac{c^2}{2}\sum p_k^{-2\sigma}$
(to leading order, from the $B^2$-norm of~$D$).
The value distribution is \emph{isotropic} in the complex plane,
as for $\log\zeta$ via Bohr--Jessen~\cite{BohrJessen1932}.

The quantitative difference is dramatic: at $\sigma = 0.75$, the
RMS of $\log(\Xi_c/\Xi_0)$ is approximately $0.046$, while
$\log\zeta$ has RMS $\approx 0.88$.
The construction's value distribution is confined to a bounded disk,
whereas $|\log\zeta(\sigma+it)|$ is unbounded.
This means $\Xi_c/\Xi_0$ \emph{cannot} satisfy Voronin
universality~\cite{Voronin1975}, and its moments do not exhibit the
$(\log T)^{k^2}$ growth of~$\zeta$
(cf.~\cite{KeatingSna2000}).  The value distribution is
isotropic but \emph{concentrated}---another manifestation of the
missing arithmetic.

\begin{figure}[ht]
\centering
\includegraphics[width=\textwidth]{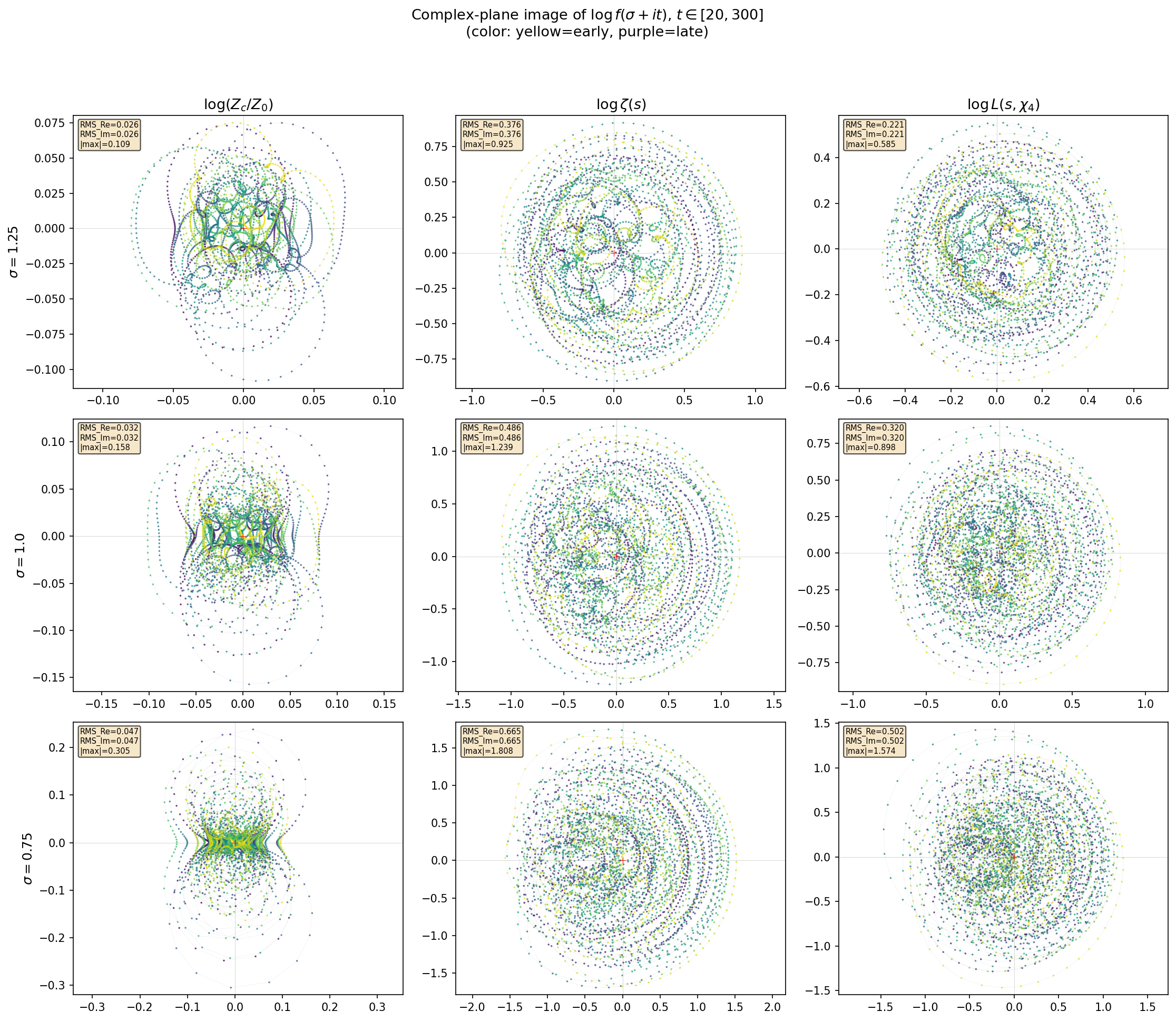}
\caption{Complex-plane image of $\log f(\sigma+it)$ for
  $t\in[20,300]$.
  Left: $\log(\Xi_c/\Xi_0)$; center: $\log\zeta$; right:
  $\log L(s,\chi_4)$.
  Rows: $\sigma=1.25, 1.0, 0.75$.  Color encodes $t$
  (yellow$=$early, purple$=$late).  Note the ${\sim}20{\times}$
  difference in scale between the left column and the
  others.}\label{fig:value-dist}
\end{figure}

\begin{figure}[ht]
\centering
\includegraphics[width=0.85\textwidth]{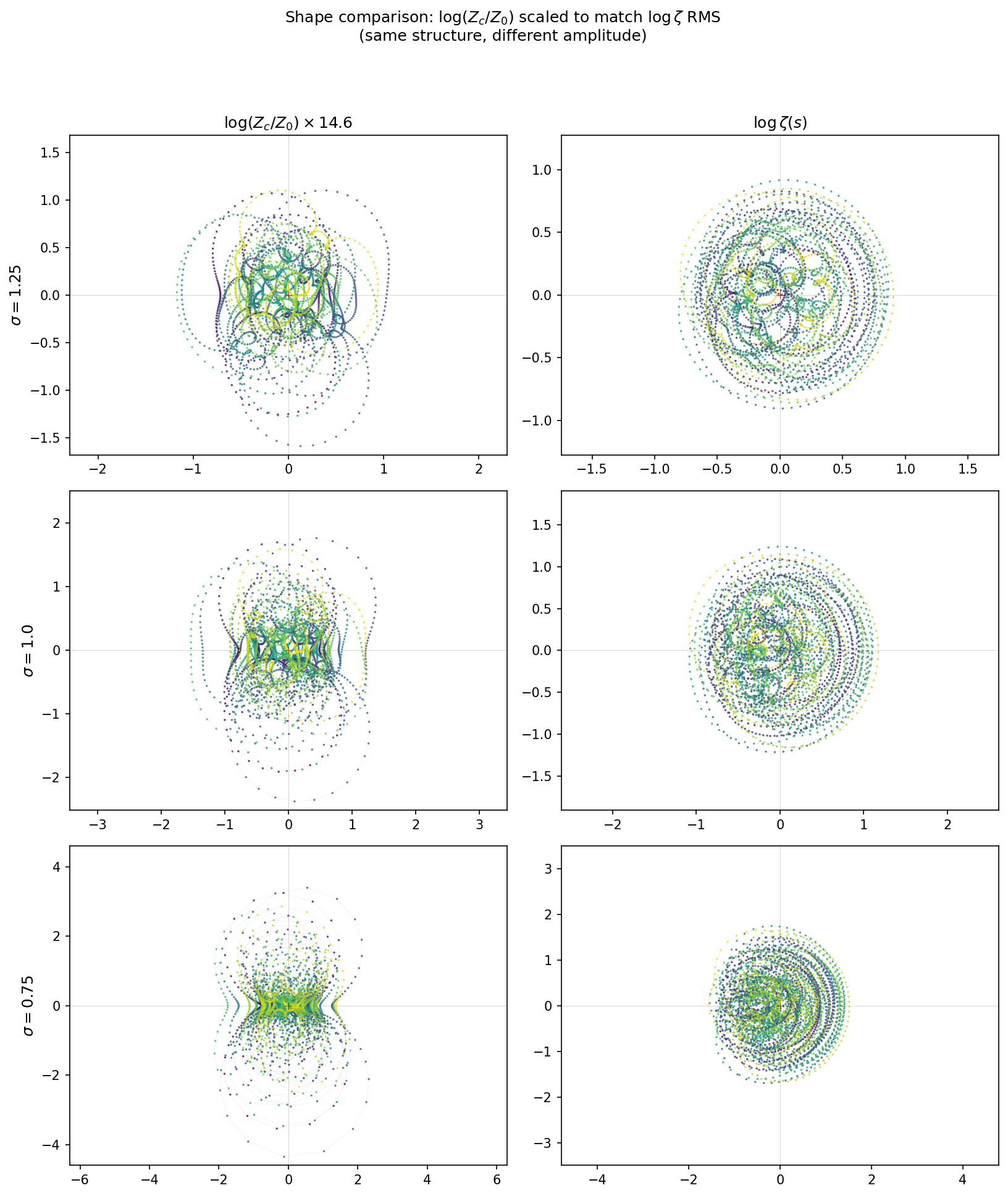}
\caption{Shape comparison: $\log(\Xi_c/\Xi_0)$ rescaled to match
  the RMS of $\log\zeta$ (left) vs.\ $\log\zeta$ (right).  The
  rescaled curve fills a \emph{banded} region (reflecting the
  coherent trigonometric structure) rather than the isotropic spirals
  of $\log\zeta$ (reflecting multiplicative independence of
  primes).}\label{fig:value-normalized}
\end{figure}

\FloatBarrier

\subsection{Numerical evidence for the linearization condition}
\label{sec:lin-evidence}

Lemma~\ref{lem:lin} establishes that
$\sum_{n=1}^\infty E(\gamma_n^{(0)})^2/n^2 < \infty$.  We verify
key aspects numerically with $K=80$ terms and $N = 5000$ zero
positions.

\textbf{Mean square and Weyl equidistribution.}
The empirical mean square agrees with the Weyl prediction:
\begin{center}\small
\begin{tabular}{rrrr}
\hline
$N$ & $\frac{1}{N}\sum_{n=1}^N E(\gamma_n^{(0)})^2$ &
  $\frac{1}{2\pi^2}\sum_{k=1}^K p_k^{-1}$ & ratio \\
\hline
100 & 0.0688 & 0.0654 & 1.053 \\
500 & 0.0668 & 0.0654 & 1.022 \\
1000 & 0.0649 & 0.0654 & 0.994 \\
5000 & 0.0648 & 0.0654 & 0.991 \\
\hline
\end{tabular}
\end{center}
The ratio stabilizes near~$1$, consistent with joint
equidistribution of the $K$ phases
$\{\omega_k\gamma_n^{(0)}+\phi_k\}$ modulo~$2\pi$.

\textbf{Spectral identity remainder.}
The remainder $R(t)$ in the spectral identity
$\re\log(\Xi_c/\Xi_0)(\sigma\!+\!it)
= -c\sum p_k^{-\sigma}\cos(\omega_k t+\phi_k) + R(t)$ has:
\begin{center}\small
\begin{tabular}{rrrr}
\hline
$\sigma$ & $\operatorname{mean}(R)$ & $\operatorname{std}(R)$ &
  $\max|R|$ \\
\hline
1.50 & $-0.015$ & $0.021$ & $0.081$ \\
1.10 & $-0.016$ & $0.030$ & $0.122$ \\
1.01 & $-0.016$ & $0.034$ & $0.136$ \\
1.00 & $-0.016$ & $0.034$ & $0.138$ \\
\hline
\end{tabular}
\end{center}
The remainder is $O(c)$ as claimed, and the linearization
contribution ($\sim 10^{-5}$) is negligible compared to the
constant and lower-limit terms.

\section{Open Problems and Concluding Remarks}\label{sec:conclusion}

The function~$\Xi_c$ provides an explicit entire function of
order~$1$ that provably satisfies an analogue of RH (all zeros on
$\sigma = 1/2$) together with the sharp bounded/unbounded transition
at $\sigma = 1$ characteristic of~$\zeta$.  The ratio $\Xi_c/\Xi_0$
is, to leading order, a Dirichlet series with generalized Euler
product (Proposition~\ref{prop:ratio}), providing a concrete object
whose zeros are provably on the critical line and whose analytic
properties can be studied unconditionally.

The construction demonstrates that the $\sigma = 1$ transition is
driven entirely by the marginal convergence/divergence of the
Dirichlet series $D(s) = \sum e^{-ik\theta}p_k^{-s}$ at
$\sigma = 1$, in direct analogy with
$\log\zeta(s) = P(s) + O(1)$~\cite[\S3.2]{Titchmarsh1986}.  In the
context of the Selberg class, this shows that an RH-like zero
distribution is compatible with the $\sigma = 1$ transition---the two
properties reinforce rather than obstruct each other.

Several directions for further investigation suggest themselves.

\textbf{GUE statistics.}
The zeros of~$Z_c$ are super-rigid (nearly crystalline spacing,
\S\ref{sec:stats}), far from the GUE statistics of~$\zeta$.
What additional structure on the perturbation would force GUE pair
correlation onto the zeros of~$\Xi_c$?

\textbf{Pointwise bounds on $E$.}
The mean square~\eqref{eq:E-meansq} shows that
$|E(\gamma_n^{(0)})| \sim \sqrt{\log\log n}$ in RMS, but
the worst-case bound~\eqref{eq:D-bound} gives only
$O(\gamma_n^{1/2+\varepsilon})$.  Numerically,
$|E(\gamma_n^{(0)})| \le 1.05$ for all computed~$n$
(\S\ref{sec:lin-evidence}).  Proving that
$|E(\gamma_n^{(0)})| = O(\sqrt{\log\log n})$ pointwise
(matching the RMS rate) would sharpen the linearization
estimate and clarify when the perturbations first exceed the
zero spacing.

\textbf{Growth rate on $\sigma = 1$.}
The Phragm\'en--Lindel\"of argument gives
$\sup_t|Z_c(1+it)| = +\infty$ but not the rate.
For~$\zeta$, $|\zeta(1+it)| = \Omega(\log\log t)$
\cite[Thm.~8.5]{Titchmarsh1986}.
Does $|Z_c(1+it)|$ achieve the same rate?

\textbf{An actual Dirichlet series.}
The ratio $\Xi_c/\Xi_0$ is \emph{approximately} a Dirichlet
series with Euler product (Proposition~\ref{prop:ratio}), but
the bounded correction factor~$M(s)$ prevents it from being
one exactly.  The base function~$\Xi_0$ is not a Dirichlet
series either (its Fourier exponents $\log\gamma_n^{(0)}$ are
not of the form $\log n$).  Can one unconditionally construct a
Dirichlet series $\sum a_n\,n^{-s}$ (convergent for
$\sigma > 1/2$) with all zeros provably on $\sigma = 1/2$ and
the $\sigma = 1$ transition?  The Davenport--Heilbronn
example~\cite{DavenportHeilbronn1936} shows that a functional
equation alone does not suffice; some form of multiplicative
structure appears necessary.

\section*{Acknowledgments}

This paper was developed through a collaboration between the
author and Claude (Anthropic).  The author posed the original
question (whether one can construct an explicit entire function
with the Riemann Hypothesis and the $\sigma = 1$ transition),
proposed the unit-disc analogy that motivated the Hadamard
product approach, identified that a family result is insufficient
(insisting on a single function), and formulated the
equidistribution criterion for unboundedness.  Claude proposed
the Hadamard product construction, developed the spectral
identity and its proof, found the dyadic large-sieve argument
that resolved the linearization convergence, suggested the
$\alpha_k = p_k$ switch that yields the Dirichlet series
structure, and wrote the paper.

\end{document}